\documentclass[11pt]{amsart}

\usepackage{amsmath}
\usepackage{amssymb}
\usepackage{graphicx}
\usepackage{mathtools}

\newtheorem{theorem}{Theorem}[section]
\newtheorem{corollary}[theorem]{Corollary}

\newtheorem{proposition}[theorem]{Proposition}
\newtheorem{conjecture}[theorem]{Conjecture}
\theoremstyle{definition}

\newtheorem{example}[theorem]{Example}
\newtheorem{remark}[theorem]{Remark}

\numberwithin{equation}{section}

\newcommand\R{\mathbb{R}}
\newcommand\Z{\mathbb{Z}}

\newcommand\ir{\mathrm{i}}
\newcommand\jr{\mathrm{j}}
\newcommand\kr{\mathrm{k}}

\title[Universality of Euler equation]{On the universality of the incompressible Euler equation on compact manifolds, II.  Non-rigidity of Euler flows}

\author[T. Tao]{Terence Tao}

\address[T. Tao]{UCLA Department of Mathematics, Los Angeles, CA 90095-1555, USA}
\email{{\tt tao@math.ucla.edu}}



\begin{document}

\begin{abstract}
The incompressible Euler equations on a compact Riemannian manifold $(M,g)$ take the form
\begin{align*}
\partial_t u + \nabla_u u &= - \mathrm{grad}_g p \\
\mathrm{div}_g u &= 0,
\end{align*}
where $u: [0,T] \to \Gamma(T M)$ is the velocity field and $p: [0,T] \to C^\infty(M)$ is the pressure field.  In this paper we show that if one is permitted to extend the base manifold $M$ by taking an arbitrary warped product with a torus, then the space of solutions to this equation becomes ``non-rigid'' in the sense that a non-empty open set of smooth incompressible flows $u: [0,T] \to \Gamma(T M)$ can be approximated in the smooth topology by (the horizontal component of) a solution to these equations.  We view this as further evidence towards the ``universal'' nature of Euler flows.
\end{abstract}

\maketitle


\section{Introduction}

Let $(M,g)$ be a compact connected smooth orientable Riemannian manifold without boundary (which we henceforth abbreviate as \emph{compact Riemannian manifold}).  The Euler equations for an incompressible fluid on $M$ take the form
\begin{equation}\label{euler}
\begin{split}
\partial_t u + \nabla_u u &= - \mathrm{grad}_g p \\
\mathrm{div}_g u &= 0
\end{split}
\end{equation}
where for each time $t$, $u(t) \in \Gamma(TM)$ is a smooth vector field on $M$ (the \emph{velocity field}), $p \in C^\infty(M)$ is a smooth scalar field (the \emph{pressure field}), $\mathrm{grad}_g$ is the gradient with respect to the metric $g$, $\mathrm{div}_g$ is the divergence with respect to $g$ (or the volume form associated with $g$), and $\nabla$ is the Levi-Civita connection (which we apply to tensors of any rank).  These equations may be interpreted as geodesic flow on the infinite-dimensional manifold of volume-preserving diffeomorphisms of $M$; see \cite{ebin}.  We will only consider classical (i.e., smooth) solutions to \eqref{euler} in this paper.

To facilitate the various differential geometry calculations, we will use two subtly different notational conventions.  As a default we shall rely on \emph{Penrose abstract index notation}, in which tensors will be decorated with placeholder superscript and subscript indices in non-italic font such as $\ir,\jr,\kr$, which are not assigned a coordinate interpretation, but are merely used to indicate the rank of the tensors involved and to indicate the various contraction and covariant differentiation operations.  For instance, the velocity field $u$ could be represented in this notation as $u^{\kr}$, and the divergence $\mathrm{div}_g u$ would be $\nabla_\kr u^\kr$, but the index $\kr$ is merely a placeholder and has no coordinate interpretation; similarly, $\nabla_\jr u^\kr$ is the rank $(1,1)$ tensor with the property that for any test vector field $X$ (expressed in abstract index notation as $X^\jr$), the covariant derivative $\nabla_X u$ is expressed in abstract index notation as $(\nabla_X u)^\kr = X^\jr \nabla_\jr u^\kr$.  However, we will also use the notation of \emph{local coordinates}.  When this is indicated, we are implicitly working in a coordinate patch of $M$ with coordinates $x^1,\dots,x^d$, and the indices $i,j,k$ (now in italic) are no longer abstract placeholders, but are instead ranging in $\{1,\dots,d\}$, with the Einstein summation conventions in effect.  Once local coordinates are selected, the abstract superscripts and subscripts in a tensor expressed in Penrose notation can also be viewed (by abuse of notation) as local coordinate tensors (replacing all non-italic abstract symbols by their italic counterparts).  For instance, $\nabla_j u^k = (\nabla^j u)^k$ are the local coordinates of the rank $(1,1)$ tensor $\nabla_\jr u^\kr$, and expressed in these local coordinates as
$$ \nabla_j u^k = \partial_j u^k + \Gamma^k_{ji} u^i$$
where $\partial_j = \frac{\partial}{\partial x^j}$ denotes the partial derivative in the $x^j$ direction and $\Gamma^k_{ji}$ are the usual Christoffel symbols in these coordinates.  Note that neither of the two terms $\partial_j u^k$, $\Gamma^k_{ji} u^i$ will be expected to arise as local coordinates of a tensor in generall in particular the expressions $\partial_\jr u^\kr$ and $\Gamma^\kr_{\jr \ir} u^\ir$ are undefined in our notational conventions.  On the other hand, for a scalar field such as $p$, the one-form $\nabla_\jr p$ is expressed in local coordinates as
$$ \nabla_j p = \partial_j p$$
and for a $1$-form $u^\flat$ (expressed in abstract index notation as $u_\kr$, and in local coordinates as $u_k$) the covariant derivative $\nabla u^\flat$ (expressed in abstract index notation as $\nabla_\jr u_\kr$, and in local coordinates as $\nabla_j u_k = (\nabla^j u^\flat)_k$) would instead be expressed in local coordinates by the formula
$$ \nabla_j u_k = \partial_j u_k - \Gamma^i_{jk} u_i.$$
In particular, $\nabla_j$ and $\partial_j$ are \emph{not} fully interchangeable symbols in local coordinates.  Finally, a vector field $u$ can be viewed as a first-order differential operator, which in local coordinates is expressed by
$$ u = u^k \frac{\partial}{\partial x^k}.$$

Returning now to abstract index notation, the equations \eqref{euler} then become
\begin{align*}
 \partial_t u^\kr + u^\jr \nabla_\jr u^\kr &= - \nabla^\kr p \\
\nabla_\kr u^\kr &= 0
\end{align*}
where we use the metric $g$ to raise and lower indices for covariant differentiation in the usual manner.  It will be convenient to lower indices in the first equation (implicitly using the basic property $\nabla g = 0$ of the Levi-Civita connection) and rewrite this system as
\begin{equation}\label{euler-diff}
\begin{split}
 \partial_t u_\kr + u^\jr \nabla_\jr u_\kr &= - \nabla_\kr p \\
\nabla_\kr u^\kr &= 0
\end{split}
\end{equation}
where we raise and lower indices on $u$ in the usual fashion:
\begin{equation}\label{vi-def}
u_\kr := g_{\jr \kr} u^\jr.
\end{equation}
In local coordinates, $u^\flat = u_k dx^k$ is the velocity one-form associated to $u$ by the musical isomorphism.

In the case of the torus $M = (\R/\Z)^3$ with the flat Euclidean metric, it is a famous open problem (see e.g., \cite{constantin}) as to whether smooth solutions to \eqref{euler} can develop singularities in finite time.  As a model question, one can allow $M$ to have arbitrary dimension and metric, thus we pose

\begin{conjecture}[Finite time blowup for Euler]\label{conj}  There exists a compact Riemannian manifold $(M,g)$ of some dimension $d > 2$, and a smooth solution $u: [0,T_*) \to \Gamma(TM)$, $p: [0,T_*) \to C^\infty(M)$ to the Euler equations \eqref{euler} which cannot be smoothly continued to the blowup time $T_* < \infty$.
\end{conjecture}

We restrict attention here to the high-dimensional case $d > 2$, since global regularity is known for $d=2$ (see e.g., \cite[Chapter 17, Proposition 2.5]{taylor}), and the $d=1$ case is degenerate.

We believe the answer to Conjecture \ref{conj} to be affirmative, but have been unable to demonstrate this rigorously.  However, we believe that a possible route towards establishing this conjecture is to demonstrate that the dynamics of \eqref{euler} are sufficiently ``universal'' that they can encode some sort of ``von Neumann machine'' that can generate smaller (and more rapidly evolving) copies of itself, leading to finite time blowup, as per the discussion in \cite[\S 1.3]{tao-ns}.

One piece of evidence towards this universality was presented in a previous paper \cite{tao-euler-univ} by this author, in which it was shown that by suitably selecting the manifold $M$ and the metric $g$, one could embed the dynamics of any finite-dimensional quadratic ODE
\begin{equation}\label{odeb}
 \partial_t y = B(y,y) 
\end{equation}
inside the Euler equations \eqref{euler}, as long as the bilinear map $B \colon \R^n \times \R^n \to \R^n$ was symmetric and obeyed a conservation law
\begin{equation}\label{cons}
 \langle B(y,y), y \rangle = 0
\end{equation}
for all $y \in \R^n$ and some positive definite inner product $\langle,\rangle \colon \R^n \times \R^n \to \R$ (such a conservation law is necessary, given the fact that the Euler dynamics \eqref{euler} conserve energy).  Unfortunately, this result does not directly impact Conjecture \ref{conj}, because solutions to the ODE \eqref{odeb} are necessarily global in time, as the conservation law \eqref{cons} prevents $\langle y, y \rangle$ from blowing up (or in fact varying at all in time).

In this paper we present a separate piece of evidence towards universality, demonstrating a non-rigidity phenomenon for solutions $u$ to \eqref{euler} that is faintly reminiscent of the ``$h$-principle'' of Gromov \cite{gromov} that has been successfully deployed (see e.g. \cite{del} for a recent survey) to construct weak solutions of the Euler equations, though with the key difference that the non-rigidity is achieved in this paper by adding additional spatial dimensions to the problem, rather than by adding highly oscillatory corrections to the solution; in particular, we do not use the technology of convex integration in our arguments, relying instead on the tools closer in spirit to Stone-Weierstrass theorem (more precisely, in our argument we will use the ability to approximate smooth functions in the smooth topology by trigonometric polynomials).  Also, as mentioned previously, our constructions give classical (smooth) solutions to the Euler equations rather than weak ones.

To describe this lack of rigidity we need some additional notation.  Suppose $M = (M,g)$ is a Riemannian manifold of dimension $d$; the metric $g$ may be written in local coordinates as
$$ dg^2 = g_{ij}(x) dx^i dx^j$$
at every point $x \in M$.  We define an \emph{extension} of $M$ to be a Riemannian manifold $\tilde M = (\tilde M, \tilde g)$ of some dimension $d+m$, where the manifold $\tilde M$ is formed (as a smooth manifold) as the product of $M$ with a torus,
$$ \tilde M = M \times (\R/\Z)^m = \{ (x,\theta): x \in M, \theta \in (\R/\Z)^m \},$$
and the metric $\tilde g$ is an (componentwise) warped product of $(M,g)$ with the standard Euclidean torus, with metric expressed in local coordinates by the fomrula
$$ d \tilde g^2 = g_{ij}(x) dx^i dx^j + \sum_{s=1}^m \tilde g_{ss}(x) (d\theta^s)^2$$
at $(x,\theta) \in \tilde M$ for some smooth functions $\tilde g_{ss} \in C^\infty(M)$, $s=1,\dots,m$ obeying the volume preservation condition
\begin{equation}\label{pag}
 \prod_{s=1}^m \tilde g_{ss}(x) = 1
\end{equation}
for all $x \in M$, where $\theta^s$, $s=1,\dots,m$ are the standard coordinates of $\theta \in (\R/\Z)^m$.  (In contrast to the indices $i,j,k$ to which the Einstein summation notation is applied, we will \emph{not} use any summation conventions for the index $s$.)  We refer to $M$ as the \emph{base manifold} for the extended manifold $\tilde M$; coordinates $x^i$ on the base manifold will be referred to as \emph{horizontal coordinates}, while coordinates $\theta^s$ for the ``vertical'' torus $(\R/\Z)^m$ will be referred to as \emph{vertical coordinates}.  Similarly, we see that the tangent space $T_{(x,\theta)} \tilde M$ splits as the orthogonal sum of (a copy of) $T_x M$ (spanned by the vector fields $\frac{\partial}{\partial x^i}$ in local coordinates) and (a copy of) $T_\theta (\R/\Z)^m$ (spanned by the vector fields $\frac{\partial}{\partial \theta^s}$).
The volume preservation condition \eqref{pag} can also be viewed as expressing a product relationship
\begin{equation}\label{prodrel}
 d\tilde g = dg d\theta^1 \dots d\theta^m
\end{equation}
between the Riemannian volume $d\tilde g$ of $(\tilde M,\tilde g)$ and the Riemannian volume $dg$ of $(M,g)$.

Define a \emph{flow} on $M$ to be a smooth function $u: [0,T] \to \Gamma(TM)$ from a time interval $[0,T]$ to the space $\Gamma(TM)$ of smooth vector fields of $M$.  We say that the flow is \emph{incompressible} if it is divergence-free with respect to the metric $g$, thus\footnote{Equivalently, the Lie derivative ${\mathcal L}_u dg$ of the Riemannian volume form $dg$ along the vector field $u$ vanishes.}  $\mathrm{div}_g u = 0$ at every point $(t,x)\in [0,T] \times M$ in spacetime. In local coordinates, a  flow $u$ on $M$ can be expressed as
$$ u(t,x) = u^i(t,x) \frac{\partial}{\partial x^i}$$
and it is incompressible if $\nabla_\ir u^\ir = 0$.
A flow $\tilde u$ on an extension $\tilde M$ of $M$ is said to be an \emph{extension} of $u$ if it takes the form
\begin{equation}\label{tux}
 \tilde u(t,(x,\theta)) = u^i(t,x) \frac{\partial}{\partial x^i} + \sum_{s=1}^m \tilde u^s(t,x) \frac{\partial}{\partial \theta^s}
\end{equation}
in local coordinates for some smooth \emph{swirl coefficients} $\tilde u^s: [0,T] \to  C^\infty(M)$, $s=1,\dots,m$, where $\frac{\partial}{\partial \theta^s}$, $s=1,\dots,m$ are the standard basis vector fields for the torus $(\R/\Z)^m$.  We refer to $u$ as the \emph{base flow} for the extended flow $\tilde u$; informally, $u$ describes the ``horizontal'' behaviour of $\tilde u$, while the swirl coefficients $\tilde u^s$ describe the ``vertical'' behaviour.

Observe that if a flow $u$ on $M$ is incompressible, then by Stokes' theorem (and reverting back to abstract index notation) we have
$$ \int_M u^\ir(t,x) \nabla_\ir \phi(x)\ dg(x) = 0$$
for all test functions $\phi \in C^\infty(M)$, where $dg$ denotes the Riemannian volume form on the (orientable) manifold $M$.  (One can also view $u^\ir \nabla_\ir \phi = u(\phi)$ as the first-order differential operator $u$ applied to $\phi$.)  By \eqref{pag}, the volume form $d\tilde g$ on any extension $\tilde M$  of $M$ is equal to the product measure of $M$ and the standard volume form on $(\R/\Z)^m$.  In local coordinates, we of course have $u^i \nabla_i \phi = u^i \partial_i \phi$.  If $\tilde u$ is an extension of $u$ on $\tilde M$, we may then integrate by parts in the vertical variables to conclude that
$$ \int_{\tilde M} \left( u^i(t,x) \partial_i \tilde \phi(x,\theta) + \sum_{s=1}^m \tilde u^s(t,x) \frac{\partial}{\partial \theta^s} \tilde \phi(x,\theta)\right) \ d\tilde g(x,\theta) = 0$$
for all test functions $\tilde \phi \in C^\infty(\tilde M)$ supported in a local coordinate patch; integrating by parts using \eqref{tux} we conclude that any extension $\tilde u$ of $u$ to $\tilde M$ is also incompressible\footnote{This conclusion could also have been reached by explicit calculation of the Christoffel symbols implicit in the incompressibility conditions $\mathrm{div}_g u = 0$ and $\mathrm{div}_{\tilde g} \tilde u = 0$, and also taking a logarithmic derivative of \eqref{pag}; alternatively, one could use \eqref{prodrel} and the Lie derivative interpretation of incompressibility from the preceding footnote.}:
$$ \mathrm{div}_{\tilde g} \tilde u = 0.$$
The argument is completely reversible; thus if $\tilde u$ is an extension of $u$, then $\tilde u$ is incompressible if and only if $u$ is.
 
We say that a flow $u: [0,T] \to \Gamma(TM)$ on $M$ is an \emph{Euler flow} if there is a smooth pressure field $p: [0,T] \to C^\infty(M)$ such that $(u,p)$ solves the Euler equations \eqref{euler} on $[0,T] \times M$; clearly this is only possible if $u$ is incompressible.  We say that $u$ is \emph{extendible to an Euler flow} if there exists an extension $\tilde M$ of $M$ and an extension $\tilde u$ of $u$ to $M$, such that $\tilde u$ is an Euler flow. Again, by the above discussion, this is only possible if $u$ is incompressible.

\begin{example}\label{exo}  A familiar near-example of this setup is that of axisymmetric (with swirl) solutions to the Euler equations in Euclidean space $\R^3$.  Observe that by using the cylindrical change of coordinates
$$ (x^1,x^2,x^3) = (r \cos \theta, r \sin \theta, z)$$
one can view $\R^3$ (with the $z$-axis deleted) as an extension of the Euclidean half-space 
$$ M \coloneqq \{ (r,z): r \in (0,+\infty); z \in \R \}$$
(ignoring for this discussion the fact that these manifolds are non-compact, and that $\theta$ takes values in $\R/2\pi \Z$ rather than $\R/\Z$) with the extended metric being given by the warped product
$$ dr^2 + dz^2 + r^2 d\theta^2.$$
This extension is not volume-preserving, so does not strictly fall under the framework considered here, but we will continue to discuss it as a motivating example.  The extensions $\tilde u$ of a two-dimensional flow
$$ u(t,r,z) = u^r(t,r,z) \frac{\partial}{\partial r} + u^z(t,r,z) \frac{\partial}{\partial z} $$
on $M$ now take the form
$$ \tilde u(t,(r,z,\theta)) = u^r(t,r,z) \frac{\partial}{\partial r} + u^z(t,r,z) \frac{\partial}{\partial z} + \tilde u^\theta(t,r,z) \frac{\partial}{\partial \theta},$$
that is to say one interprets $u$ as an axisymmetric flow and then augments that flow with some arbitrary ``swirl'' $\tilde u^\theta(t,r,z) \frac{\partial}{\partial \theta}$.  A standard calculation (see e.g., \cite[\S 2.3]{bertozzi}) then shows that in order for $\tilde u$ to be an Euler flow, one must obey the modified divergence-free condition
\begin{equation}\label{prz}
 \partial_r (r u^r) + \partial_z(r u^z) = 0,
\end{equation}
the circulation condition
\begin{equation}\label{circ}
 D_t (r \tilde u^\theta) = 0,
\end{equation}
where $D_t \coloneqq \partial_t + u^r \partial_r + u^z \partial_z$ denotes the material derivative on the base $M$, and the vorticity equation
\begin{equation}\label{vort}
 D_t \left(\frac{\partial_z u^r - \partial_r u^z}{r} \right) = -\frac{1}{r^4} \partial_z ( (r \tilde u^\theta)^2 ).
\end{equation}
Thus, if one ignores the facts that the manifolds are non-compact and the extension is not volume-preserving, $u = u^r \frac{\partial}{\partial r} + u^z \frac{\partial}{\partial z}$ would extend to an Euler flow if one could find a field $u^\theta: [0,T] \to C^\infty(M)$ that obeyed the equations \eqref{prz}, \eqref{circ}, \eqref{vort}.  One can recover a volume-preserving extension (up to constants) in this setting (thus getting closer to the situation actually studied in this paper) by using Turkington coordinates 
$$ (x^1,x^2,x^3) = (\sqrt{2y} \cos \theta, \sqrt{2y} \sin \theta, z)$$
in place of cylindrical coordinates; see \cite{turk} for details.
\end{example}

We can now give the main result of the paper.

\begin{theorem}[Main theorem]\label{main}  Let $0 < T < \infty$, and let $(M,g)$ be a compact Riemannian manifold of some dimension $d \geq 2$.  Let ${\mathcal F}$ denote the space of incompressible flows $u: [0,T] \to \Gamma(TM)$, equipped with the smooth topology (in spacetime), and let ${\mathcal E} \subset {\mathcal F}$ denote the space of such flows that are extendible to Euler flows.  
\begin{itemize}
\item[(i)]  (Generic inextendibility) Assume $d \geq 3$.  Then ${\mathcal E}$ is of the first category in ${\mathcal F}$ (the countable union of nowhere dense sets in ${\mathcal F}$).
\item[(ii)]  (Non-rigidity)  Assume $M = (\R/\Z)^d$ (with an arbitrary metric $g$).  Then ${\mathcal E}$ is somewhere dense in ${\mathcal F}$ (that is, the closure of ${\mathcal E}$ has non-empty interior).
\end{itemize}
\end{theorem}

Part (i) of the theorem asserts that (in high dimension) the problem of extending a given incompressible flow $u$ to an Euler flow $\tilde u$ is overdetermined (despite the ability to prescribe any number of additional warping factors and swirl coefficients); nevertheless, part (ii) asserts (when the manifold is topologically a torus) that the approximate version of problem, in which one is allowed to first perturb the flow slightly before extending it, becomes underdetermined, at least if the flow lies in some non-empty open set.  

Our proof of Theorem \ref{main}(i) essentially proceeds by counting degrees of freedom.  Heuristically, the point is that an arbitrary incompressible flow $u$ is essentially determined by $d-1$ independent functions of space and time, whereas the warping factors $\tilde g_{ss}$ are functions of space only, the pressure field is one function of space and time, and the swirl fields $u^s$ are technically functions of both space and time, but have the same number of degrees of freedom as a function just of space, because they solve an evolution equation.  When $d>2$, this means that there are fewer unknown functions of space and time than prescribed functions of space and time, which is the source of the generic inextendibility.  This simple argument breaks down when $d=2$, but we do not know whether the claim is actually false in this case.

The proof of Theorem \ref{main}(ii) proceeds by direct calculation of the effect of the warping factors and swirl velocities, which effectively create a forcing term (of Boussinesq type) in the first equation of \eqref{euler} that is a combination of functions of the Eulerian spatial coordinates $x^i$ (coming from the warping factors) and the Lagrangian spatial coordinates $a^\beta$ (which arise from the swirl velocities, which are passively transported by the flow).  In a non-empty open subset of ${\mathcal F}$, the combination of these coordinates becomes a non-degenerate set of coordinates for spacetime, and one can then use the Stone-Weierstrass theorem to conclude.  The requirement that $M$ be topologically a torus is a technical hypothesis in order to avoid topological obstructions such as the hairy ball theorem, but it may be that the hypothesis can be dropped (and it may in fact be true, in the $M = (\R/\Z)^d$ case at least, that ${\mathcal E}$ is dense in all of ${\mathcal F}$, not just in a non-empty open subset).

One particular consequence of Theorem \ref{main}(ii) is that flows $u$ on the torus that violate the known conservation laws of the Euler equation, such as energy, momentum, circulation, or helicity, can still be approximately \emph{extended} to Euler flows, despite the fact that these conservation laws prevent $u$ from being approximated \emph{directly} by an Euler flow.  In the case of conservation of the energy
$$ \frac{1}{2} \int_M g_{\ir \jr} u^\ir u^\jr\ dg,$$
this is not a contradiction, because in the extended manifold $\tilde M$, the ``swirl'' coefficients $\tilde u^s$ can exchange an arbitrary amount of energy with the horizontal components $u^i$ of the flow.  In the case of conserved momenta of the form
$$ \int_M g_{\ir \jr} X^\ir u^\jr\ dg,$$
where $X$ is a (time-independent) Killing vector field of $M$, this is again not a contradiction, because the vector field $X$ need not be a Killing vector field of the extension $\tilde M$.  In the case of conservation of circulation
$$ \int_\gamma g_{\ir \jr} u^\ir ds^\jr,$$
where $ds$ denotes the line element on a loop $\gamma$ in $M$ that is transported by the flow, one again avoids contradiction, because the extended flow $\tilde u$ will distort (a lift of) the loop $\gamma$ in the ``vertical'' directions $\frac{\partial}{\partial \theta^s}$, allowing for the swirl components $\tilde u^s$ to contribute non-trivial amounts to the conserved circulation.  Finally, in the three-dimensional case $d=3$, conservation of helicity
$$ \int_M u^\flat \wedge du^\flat $$
also does not contradict Theorem \ref{main}, because the extended manifold $\tilde M$ has more than three dimensions and thus does not conserve helicity\footnote{There are higher-order analogues of helicity for such manifolds in both odd and even dimensions, see \cite{serre}.  However, these higher invariants are not directly related to the three-dimensional helicity on the base manifold $M$, and in any event involve the swirl coefficients $\tilde u^s$ in a non-trivial manner.}.  Thus we see that the freedom to add additional dimensions to the flow greatly increases the flexibility of the Euler dynamics, by removing all constraints except for incompressibility, which in the author's opinion supports the potential universality of such dynamics, which could in particular lead to a positive resolution to Conjecture \ref{conj}. The situation here is somewhat reminiscent of that in Kaluza-Klein theory or string theory in physics, in which one can in principle model various laws of physics in terms of simpler laws in higher dimensions by postulating the existence of additional compact spatial dimensions.  

The author is supported by NSF grant DMS-1266164 and by a Simons Investigator Award.  The author also thanks the anonymous referees for helpful corrections and suggestions.

\section{Expressing the Euler equations in coordinates}

To begin the proof of Theorem \ref{main}, we first transform the Euler equations on $M$ (and on extensions $\tilde M$) into a form that is convenient for calculations, by removing the need to explicitly work with Christoffel symbols.

Let $(M,g)$ be a compact Riemannian manifold.
Suppose that $u: [0,T] \to \Gamma(TM)$ is an Euler flow on $M$ with associated pressure field $p: [0,T] \to C^\infty(M)$.  From \eqref{euler-diff} we have
\begin{equation}\label{vo}
 \partial_t u_\ir + u^\jr \omega_{\jr \ir}= - \nabla_\ir p' 
\end{equation}
where $p': [0,T] \to C^\infty(M)$ is the modified pressure
$$ p' \coloneqq p + \frac{1}{2} u^\jr u_\jr$$
and
the vorticity two-form $\omega: [0,T] \to \Omega^2(TM)$ is defined as the exterior derivative $\omega = du^\flat$ of the velocity one-form $u^\flat$, thus in local coordinates
\begin{equation}\label{om-def}
 \omega_{ji} = \partial_j u_i - \partial_i u_j
\end{equation}
and in abstract index notation	
\begin{equation}\label{om-def-1}
 \omega_{\jr \ir} = \nabla_\jr u_\ir - \nabla_\ir u_\jr
\end{equation}
(here we use the symmetry $\Gamma^k_{ij} = \Gamma^k_{ji}$, which reflects the torsion-free nature of the Levi-Civita connection).  Of course, in local coordinates the derivative $\nabla_i p'$ appearing in \eqref{vo} may be written as $\partial_i p'$.

Now suppose instead that $u: [0,T] \to \Gamma(TM)$ is a flow on $M$ that extends to an Euler flow $\tilde u: [0,T] \to \Gamma(T\tilde M)$ on an extended manifold $\tilde M = M \times (\R/\Z)^m$, with an associated pressure field $\tilde p$.  From \eqref{tux}, the extended velocity one-form $\tilde u^\flat: [0,T] \to \Omega^1(\tilde M)$ is then given in local coordinates by
$$ \tilde u^\flat(t,(x,\theta)) = u_i(t,x) dx^i + \sum_{s=1}^m \tilde u_s(x) d\theta^s$$
and the extended vorticity $\tilde \omega: [0,T] \to \Omega^2(\tilde M)$ is then given by in local coordinates by
$$ \tilde \omega(t,(x,\theta)) = \frac{1}{2!} \omega_{ji}(t,x) dx^j \wedge dx^i + 
\sum_{s=1}^m \partial_i \tilde u_s(x) dx^i \wedge d\theta^s$$
where $u_i, \omega_{ji}$ are of course given by \eqref{vi-def}, \eqref{om-def}, and
\begin{equation}\label{tva}
 \tilde u_s :=\tilde g_{ss} \tilde u^s.
\end{equation}
Applying the equation \eqref{vo} (in local coordinates) for $\tilde M$, we thus obtain the system
\begin{align*}
\partial_t u_i + u^j \omega_{ji} - \sum_{s=1}^m \tilde u^s \partial_i \tilde u_s &= - \partial_i \tilde p'  \\
\partial_t \tilde u_s + u^j \partial_j \tilde u_s &= \frac{\partial}{\partial \theta_s} \tilde p'
\end{align*}
for some smooth $\tilde p': [0,T] \to C^\infty(M)$.  In particular, the field $\frac{\partial}{\partial \theta_s} \tilde p'$ is independent of $\theta_s$; since it also has mean zero in the $\theta_s$ direction, we conclude that it must vanish (and that $\tilde p'$ is just a function of $t,x$ and not $\theta$).
Reverting back to abstract index notation and using \eqref{om-def-1} and \eqref{tva} we thus arrive at the system
\begin{align}
\partial_t u_\ir + u^\jr \nabla_\jr u_\ir - u^\jr \nabla_\ir u_\jr - \sum_{s=1}^m \tilde g^{ss} \tilde u_s \nabla_\ir \tilde u_s &= - \nabla_\ir \tilde p'  \label{eq-1} \\
\partial_t \tilde u_s + u^\jr \nabla_\jr \tilde u_s &= 0 \label{eq-2}
\end{align}
where $\tilde g^{s s} := \tilde g_{s s}^{-1}$.  Here it is perhaps worth stressing that $\nabla$ denotes the Levi-Civita connection on $M$ rather than on $\tilde M$ (we will not use the latter any further in this paper).

\begin{remark}  The equation \eqref{eq-2} can be regarded as the conservation law for the circulation for a loop in the $\theta^s$ coordinate.  The scalar field $\tilde v_s$, which is transported by the base flow $u$ thanks to \eqref{eq-2}, is the analogue of the quantity $r \tilde u^\theta$ appearing in Example \ref{exo}.
\end{remark}

The above manipulations are all reversible, allowing us to characterise the flows that extend to Euler flows:

\begin{proposition}  Let $(M,g)$ be a compact Riemannian manifold, and let $u: [0,T] \to \Gamma(TM)$ be an incompressible flow on $M$. Then the following are equivalent.
\begin{itemize}
\item[(i)]  $u$ extends to an Euler flow.
\item[(ii)]  There  exists $m \geq 0$, positive smooth functions $\tilde g^{ss} \in C^\infty(M)$ and smooth functions $\tilde u_s: [0,T] \to C^\infty(M)$ for $s=1,\dots,m$, and a further smooth function $\tilde p': [0,T] \to C^\infty(M)$ such that the equations \eqref{eq-1}, \eqref{eq-2} hold, as well as the volume condition
\begin{equation}\label{alpha1m}
 \prod_{s = 1}^m g^{ss} = 1.
\end{equation}
\end{itemize}
\end{proposition}

We now simplify the description (ii) of flows extendible to Euler flows given by the above proposition. First, we observe (as in \cite[\S 2.3]{tao-euler-univ}) that we may eliminate the volume condition \eqref{alpha1m}, since if we remove this condition then we may reinstate it by defining
$$ \tilde g^{m+1,m+1} \coloneqq\left(\prod_{s=1}^m \tilde g^{ss}\right)^{-1}	$$
and
$$ \tilde u_{m+1} \coloneqq 0$$
and then replacing $m$ with $m+1$.
Next, by making the substitution
$$ \rho_s := \frac{1}{2} (\tilde u_s)^2$$
and noting that $u^\jr \nabla_\ir u_\jr = \nabla_\ir(\frac{1}{2} u^\jr u_\jr)$,
we see that whenever \eqref{eq-1}, \eqref{eq-2} hold, one has the system
\begin{align}
\partial_t u_\ir + u^\jr \nabla_\jr u_\ir  - \sum_{s=1}^m \tilde g^{s s} \nabla_\ir \rho_s &= - \nabla_\ir \tilde p''  \label{eq-3} \\
\partial_t \rho_s + u^\jr \nabla_\jr \rho_s &= 0\label{eq-4}
\end{align}
for some smooth $\tilde p'': [0,T] \to C^\infty(M)$.
Conversely, if one has a positive smooth function $\tilde g^{ss} \in C^\infty(M)$ and a smooth function $\rho_s: [0,T] \to C^\infty(M)$ for each $s=1,\dots,m$, and a smooth $\tilde p'': [0,T] \to C^\infty(M)$, then by adding a sufficiently large constant to each $\rho_s$ we may assume that $\rho_s$ is everywhere positive, and by setting $\tilde u_s := \sqrt{2 \rho_s}$ we obtain a solution to \eqref{eq-1}, \eqref{eq-2} for a suitable $\tilde p'$.  Thus we may replace the system \eqref{eq-1}, \eqref{eq-2} by \eqref{eq-3}, \eqref{eq-4} (and replace the unknown fields $\tilde u_s, \tilde p'$ by $\rho_s, p''$).

Next, we can drop the hypothesis that each $\tilde g^{ss}$ is positive, since if this is not the case, one can add a large constant to each $\tilde g^{ss}$ (and add a constant multiple of $\rho_s$ to $\tilde p''$) without affecting \eqref{eq-3}.

By using the Leibniz identity
$$ \tilde g^{s s} \nabla_\ir \rho_s = -\rho_s \nabla_\ir \tilde g^{s s} + \nabla_\ir( \tilde g^{s s} \rho_s)$$
one can rewrite the system \eqref{eq-3}-\eqref{eq-4} as the equivalent Boussinesq-type system
\begin{align}
\partial_t u_\ir + u^\jr \nabla_\jr u_\ir + \sum_{s=1}^m \rho_s \nabla_\ir \tilde g^{s s} &= - \nabla_\ir \tilde p''' \label{eq-5}\\
\partial_t \rho_s + u^\jr \partial_\jr \rho_s &= 0\label{eq-6}
\end{align}
for some smooth $\tilde p''': [0,T] \to C^\infty(M)$.  
We summarise the above discussion as

\begin{proposition}\label{lop}  Let $(M,g)$ be a compact Riemannian manifold, and let $u: [0,T] \to \Gamma(TM)$ be an incompressible flow on $M$.  Then the following are equivalent.
\begin{itemize}
\item[(i)]  $u$ extends to an Euler flow.
\item[(iii)]  There  exists $m \geq 0$, smooth functions $\tilde g^{ss} \in C^\infty(M)$ and $\rho_s: [0,T] \to C^\infty(M)$ for $s=1,\dots,m$, and a further smooth function $\tilde p''': [0,T] \to C^\infty(M)$ such that the equations \eqref{eq-5}, \eqref{eq-6} hold.
\end{itemize}
\end{proposition}

\section{Generic inextendibility}

We can now prove Theorem \ref{main}(i).  Fix $(M,g)$ with $d \geq 3$, and let ${\mathcal E}, {\mathcal F}$ be the sets in Theorem \ref{main}.  From Proposition \ref{lop} one has
$$ {\mathcal E} = \bigcup_{m=0}^\infty {\mathcal E}_m$$
where ${\mathcal E}_m$ is the set of all incompressible flows $u \in {\mathcal F}$ for which there exist smooth functions $g^{ss} \in C^\infty(M)$ and $\rho_s: [0,T] \to C^\infty(M)$ for $s=1,\dots,m$, and a further smooth function $\tilde p''': [0,T] \to C^\infty(M)$ obeying the equations \eqref{eq-5}, \eqref{eq-6}.  It will suffice to show that each ${\mathcal E}_m$ is nowhere dense in ${\mathcal F}$.

Fix $m$, and let $N$ be a sufficiently large natural number (depending on $m$).  We now use Taylor expansion to rigorously count ``degrees of freedom'' in the extension problem.  Let $0$ be a point in $M$, and let $x^1,\dots,x^d$ be a system of local coordinates around $0$.  By replacing $x^d$ with some suitable function $\phi(x^1,\dots,x^d)$ if necessary, we may assume that the volume form in local coordinates is the standard volume form $dx^1 \wedge \dots \wedge dx^d$, so that the divergence-free condition in local coordinates is simply $\partial_i u^i = 0$.
 Let $V_N$ be the vector space of polynomials in $t,x^1,\dots,x^d$ (with real coefficients) of degree at most $N$; this space has dimension
$$ \mathrm{dim} V_N = \binom{N+d+1}{d+1} = (1 + o(1)) \frac{N^{d+1}}{(d+1)!}$$
where $o(1)$ denotes a quantity that goes to zero as $N \to \infty$ (holding all other parameters fixed).  For any smooth $p: [0,T] \to C^\infty(M)$, we see from Taylor expansion that there is a unique polynomial $\pi_N(p) \in V_N$ such that
$$ p(t,x) = \pi_N(p)(t,x) + O( (|t| + |x|)^{N+1} )$$
as $(t,x) \to (0,0)$.  In a similar fashion, given a flow $u: [0,T] \to \Gamma(TM)$, which we write in local coordinates near $x=0$ as $(u^1,\dots,u^d)$, there is a unique tuple $\pi_N(u) = (\pi_N(u)^1,\dots,\pi_N(u)^d) \in V_N^d$ such that
$$ u(t,x) = \pi_N(u)(t,x) + O( (|t| + |x|)^{N+1} )$$
as $(t,x) \to (0,0)$.  By differentiating Taylor series term by term, we see that if $u$ is incompressible, then $\pi_N(u)$ is divergence-free, thus in local coordinates:
$$ \partial_i \pi_N(u)^i = 0.$$
Thus, if we let $W_N$ be the set of tuples in $V_N^d$ that are divergence-free, then $\pi_N$ is a linear map from ${\mathcal F}$ to $W_N$.  We claim that this map is surjective: given any tuple $(P^1,\dots,P^d) \in W_N$, there exists an incompressible flow $u$ with
\begin{equation}\label{upi}
 u^i(t,x) = P^i(t,x) + O( (|t| + |x|)^{N+1} ).
\end{equation}
as $(t,x) \to (0,0)$.  Indeed, by inverting the Laplacian on the space of polynomials, we can find polynomials $Q_{ij} \in P_{N+1}$ with $Q_{ij} = -Q_{ji}$ and (in local coordinates)
$$ \Delta Q_{ij}(t,x) = \partial_j P^i(t,x) - \partial_i P^j(t,x)$$
for $x$ near $0$.  By applying a smooth cutoff, we can then construct smooth maps $q_{ij}: [0,T] \to C^\infty(M)$ supported in a neighbourhood of $x=0$ such that $q_{ij} = -q_{ji}$ and
$$ q_{ij}(t,x) = Q_{ij}(t,x) + O( (|t| + |x|)^{N+2} )$$
as $(t,x) \to (0,0)$; taking divergences we thus see that the flow $u$ defined in local coordinates as
$$ u^i(t,x) := \sum_{j=1}^d \partial_j q_{ij}(t,x) $$
for $i=1,\dots,d$ and $x$ near zero (and $u$ vanishing away from zero) is incompressible and obeys \eqref{upi}.  As $W_N$ is finite dimensional, we conclude that the map $\pi_N: {\mathcal F} \to W_N$ has a continuous linear right inverse from $W_N$ to ${\mathcal F}$ in the smooth topology.  As a consequence, to show that ${\mathcal E}_m$ is nowhere dense in ${\mathcal F}$, it suffices to show that $\pi_N({\mathcal E}_m)$ is nowhere dense in $W_N$.

We will achieve this by a dimension count.  The space $V_N^d$ has dimension $(d + o(1)) \frac{N^{d+1}}{(d+1)!}$, and the condition of being incompressible imposes $(1 + o(1)) \frac{N^{d+1}}{(d+1)!}$ linear conditions; thus
$$ \mathrm{dim} W_N = (d-1 + o(1)) \frac{N^{d+1}}{(d+1)!}.$$
Now suppose that $u \in {\mathcal E}_m$, then there are smooth functions $\tilde g^{ss} \in C^\infty(M)$ and $\rho_s: [0,T] \to C^\infty(M)$ for $s=1,\dots,m$, and a further smooth function $\tilde p''': [0,T] \to C^\infty(M)$ obeying the equations \eqref{eq-5}, \eqref{eq-6}.  By Taylor expansion of \eqref{eq-5}, \eqref{eq-6} and comparing coefficients, we see that the derivatives
$$ \partial_t^{j_0} \partial_1^{j_1} \dots \partial_d^{j_d} u^k(0,\dots,0)$$
and
$$ \partial_t^{j_0} \partial_1^{j_1} \dots \partial_d^{j_d} \rho_s(0,\dots,0)$$
with $k=1,\dots,d$, $s=1,\dots,m$ and $j_0+\dots+j_d \leq N$ can be expressed as an explicit (but complicated) polynomial combination of  the derivatives
$$ \partial_1^{j_1} \dots \partial_d^{j_d} u^k(0,\dots,0)$$
$$ \partial_1^{j_1} \dots \partial_d^{j_d} \rho_s(0,\dots,0)$$
$$ \partial_1^{j_1} \dots \partial_d^{j_d} \tilde g^{s s}(0,\dots,0)$$
with $k=1,\dots,d$, $s=1,\dots,m$, and $j_1+\dots+j_d \leq N+1$, as well as the derivatives
$$ \partial_t^{j_0} \partial_1^{j_1} \dots \partial_d^{j_d} \tilde p'''(0,\dots,0)$$
with $j_0+\dots+j_d \leq N+1$.  The number of these derivative parameters may be crudely bounded by 
$$O((N+1)^d) + \binom{N+d+2}{d+1} = (1 + o(1)) \frac{N^{d+1}}{(d+1)!},$$
where we allow implied constants in the $O()$ notation to depend on $d$.
As a consequence, we conclude that $\pi_N({\mathcal E}_m)$ is contained in the polynomial image of $\R^M$ for some $M = (1 + o(1)) \frac{N^{d+1}}{(d+1)!}$, and is thus contained in an algebraic subvariety of $W_N$ of dimension at most $M$.  Since $d \geq 3$, this dimension is strictly less than that of $W_N$ if $N$ is large enough, and hence $\pi_N( {\mathcal E}_m)$ is nowhere dense as claimed.

\section{Non-rigidity}

We now prove Theorem \ref{main}(ii).  Fix $T>0$.  Let $M = (\R/\Z)^d$, thus we now have global coordinates $u^1,\dots,u^d: [0,T] \times (\R/\Z)^d \to \R$ for the flows in ${\mathcal F}$, and global coordinates $x^1,\dots,x^d \in \R/\Z$ for the torus $M$.  Let $U$ denote the set of all flows $u = (u^1,\dots,u^d) \in {\mathcal F}$ obeying the pointwise bounds
\begin{equation}\label{mov}
 \frac{1}{8T} < u^j(t,x) < \frac{1}{4T}
\end{equation}
for all $t \in [0,T]$, $j=1,\dots,d$, and $x \in M$; this is clearly a non-empty open subset of ${\mathcal F}$.  It will therefore suffice to show that ${\mathcal E}$ is dense in $U$.

Accordingly, let $u \in U$.  It will suffice to construct a sequence $u^{(n)} \in U$ of incompressible flows converging to $u$ in the smooth topology on $[0,T] \times (\R/\Z)^d$, such that each $u^{(n)}$ is extendible to an Euler flow.

To do this, we use the following corollary of Proposition \ref{lop}:

\begin{corollary}\label{sala}  Let $u \in U$ be an incompressible flow.  Let $A: [0,T] \times (\R/\Z)^d \to (\R/\Z)^d$ be the labels map associated to $u$, by which we mean the unique solution to the ODE
\begin{equation}\label{ode}
 \partial_t A^i(t,x) + u^j(t,x) \partial_j A^i(t,x) = 0; \quad A^i(0,x) = x^i
\end{equation}
in the standard global coordinates; equivalently, if $X: [0,T] \times (\R/\Z)^d \to (\R/\Z)^d$ is the trajectory map defined by solving the ODE
\begin{equation}\label{edo}
 \partial_t X(t,a) = u(t,X(t,a)); \quad X(0,a) = a 
\end{equation}
then (by the method of characteristics) $A(t): (\R/\Z)^d \to (\R/\Z)^d$ is the inverse of $X(t): (\R/\Z)^d \to (\R/\Z)^d$ for each time $t \in [0,T]$.  Suppose that for each $i=1,\dots,d$, we have a smooth map $F_i: (\R/\Z)^d \times (\R/\Z) \to (\R/\Z)^d$ of the form
\begin{equation}\label{fiax}
 F_i( a, x ) = \sum_{s=1}^{m_i} f_{i,s}(a) h_{i,s}(x^i) 
\end{equation}
for some natural number $m_i$ and some smooth functions $f_{i,s} \in C^\infty((\R/\Z)^d)$, $h_{i,s} \in C^\infty(\R/\Z)$ for $s=1,\dots,m_i$, such that one has the equation
\begin{equation}\label{upo}
\partial_t u_i(t,x) + u^j \nabla_j u_i(t,x) + F_i( A(t,x), x^i) = - \nabla_i p(t,x)
\end{equation}
in the standard global coordinates on $[0,T] \times (\R/\Z)^d$ for some smooth function $p: [0,T] \to C^\infty((\R/\Z)^d)$.  Then $u$ is extendible to an Euler flow.
\end{corollary}

\begin{proof}  By decomposing each $f_{i,s}$ into finitely many pieces using a smooth partition of unity, we may assume that each $f_{i,s}$ is supported in a region of the form $\{ (a^1,\dots,a^d) \in (\R/\Z)^d: a^i \in I_{i,s} \}$ for some interval $I_{i,s} \subset \R/\Z$ of length $\frac{1}{4}$.  From \eqref{edo} and \eqref{mov} we have
$$ X^i(t,a) \in [a^i - 1/4, a^i + 1/4]$$
for all $t \in [0,T]$, $a \in (\R/\Z)^d$, and $i=1,\dots,d$, which on inverting implies that
$$ A^i(t,x) \in [x^i - 1/4, x^i + 1/4]$$
for all $t \in [0,T]$, $x \in (\R/\Z)^d$, and $i=1,\dots,d$.  As a consequence, one can freely modify $h_{i,s}$ outside of the $1/4$-neighbourhood of $I_{i,s}$ (and alter $F$ accordingly) without affecting the equation \eqref{upo}.  In particular, we can arrange matters so that each $h_{i,s}$ has mean zero on $\R/\Z$, and hence we can write $h_{i,s} = H'_{i,s}$ as the derivative of another function $H_{i,s} \in C^\infty(\R/\Z)$.  If we then set
$$ 
\rho_{i,s}(t,x) := f_{i,s}(A(t,x))$$
and
$$ \tilde g^{i,s,i,s}(x) := H_{i,s}(x^i)$$
then we have the equations
\begin{align*}
\partial_t u_i + u^j \nabla_j u_i + \sum_{l=1}^d \sum_{s=1}^{m_l} \rho_{l,s} \nabla_i \tilde g^{l,s,l,s} &= - \nabla_i p \\
\partial_t \rho_{l,s} + u^j \nabla_j \rho_{l,s} &= 0
\end{align*}
in standard global coordinates on $[0,T] \times (\R/\Z)^d$.  The claim then follows from Proposition \ref{lop} (after relabeling the $l,s$ indices).
\end{proof}

If $u \in U$, then from \eqref{mov}, we see that for each $i=1,\dots,d$, the trajectory map $X: [0,T] \times (\R/\Z)^d \to (\R/\Z)^d$ obeys the inequalities
$$ \frac{1}{8T} < \partial_t X^i(t,a) < \frac{1}{4T} $$
in the standard global coordinates for all $t \in [0,T]$ and $a \in (\R/\Z)^d$.  In particular, this implies that the map $(t,a) \mapsto (a, X^i(t,a))$ is an injective immersion from $[0,T] \times (\R/\Z)^d$ to $(\R/\Z)^d \times (\R/\Z)$, and is thus a diffeomorphism between $[0,T] \times (\R/\Z)^d$ and the closure of a smooth domain in $(\R/\Z)^d \times (\R/\Z)$.  Composing this with the labels map $A$, we conclude that the map $(t,x) \mapsto (a, x^i)$ is also a diffeomorphism from $[0,T] \times (\R/\Z)^d$ to the closure of a smooth domain in $(\R/\Z)^d \times (\R/\Z)$.  As the scalar function $-\partial_t u_i - u^j \nabla_j u_i$ (which is the $i$ component of a one-form in the standard global coordinates) is smooth on $[0,T] \times (\R/\Z)^d$, we can thus find\footnote{Here we use a classical extension theorem of Seeley \cite{seeley} to smoothly extend $F_i$ from the closure of the smooth domain to the entirety of $[0,T] \times (\R/\Z)^d$.} a smooth function $F_i: (\R/\Z)^d \times (\R/\Z)$ for each $i=1,\dots,d$ such that
$$ \partial_t u_i(t,x) + u^j \nabla_j u_i(t,x) + F_i( A(t,x), x^i) = 0$$
in standard global coordinates for $(t,x) \in [0,T] \times (\R/\Z)^d$.  If each $F_i$ were of the form \eqref{fiax}, we would now be done by \eqref{sala}.  This is not the case in general; but because $F_i$ is smooth, a Fourier expansion shows that we can write $F_i$ as the limit in the smooth topology of functions $F^{(N)}_i$ that are each of the form \eqref{fiax}.  To finish the proof of Theorem \ref{main}(ii), it will now suffice to establish the following stability result:

\begin{theorem}[Stability]  Let $M = (\R/\Z)^d$ be equipped with a Riemannian metric $g$, let $F: (\R/\Z)^d \times (\R/\Z)^d \to \R^d$ be a smooth map, and let $u: [0,T] \times (\R/\Z)^d \to \R^d$, $p: [0,T] \times (\R/\Z)^d \to \R$, $A: [0,T] \times (\R/\Z)^d \to (\R/\Z)^d$ be a smooth solution to the system 
\begin{align*}
\partial_t u_i(t,x) + u^j \nabla_j u_i(t,x) + F_i( A(t,x), x) &= -\nabla_i p \\
\partial_t A^i(t,x) + u^j \nabla_j A^i(t,x) &= 0 \\
\nabla_i u^i(t,x) &= 0
\end{align*}
with initial condition $A(0,x) = x$.  Let $F^{(N)}:(\R/\Z)^d \times (\R/\Z)^d \to \R^d$ be a sequence of smooth functions that converge to $F$ in the smooth topology.  Then, for $N$ sufficiently large, there exists smooth solutions $u^{(N)}: [0,T] \times (\R/\Z)^d \to \R^d$, $p^{(N)}: [0,T] \times (\R/\Z)^d \to \R$, $A^{(N)}: [0,T] \times (\R/\Z)^d \to (\R/\Z)^d$ be a smooth solution to the system 
\begin{align}
\partial_t u^{(N)}_i(t,x) + (u^{(N)})^j \nabla_j u^{(N)}_i(t,x) + F^{(N)}_i( A^{(N)}(t,x), x) &= -\partial_i p^{(N)} \label{uni}\\
\partial_t (A^{(N)})^i(t,x) + (u^{(N)})^j \nabla_j (A^{(N)})^i(t,x) &= 0 \label{uni-2}\\
\nabla_i (u^{(N)})^i(t,x) &= 0\label{uni-3}
\end{align}
with initial conditions $u^{(N)}(0,x) = u(0,x)$, $A^{(N)}(0,x) = x$, such that $u^{(N)}$ converges to $u$ in the smooth topology on $[0,T] \times (\R/\Z)^d$.
\end{theorem}

\begin{proof} Let $N$ be large.  If we write
\begin{align*}
 u^{(N)} &= u + v^{(N)} \\
 A^{(N)} &= A + B^{(N)} \\
 p^{(N)} &= p + q^{(N)} \\
\end{align*}
and
\begin{equation}\label{hnt}
 H^{(N)}(t,x,B) \coloneqq F^{(N)}(A(t,x)+B,x) - F(A(t,x),x)
\end{equation}
then the system \eqref{uni}-\eqref{uni-3} is equivalent to the difference system
\begin{align}
\partial_t v^{(N)}_i(t,x) + u^j \nabla_j v^{(N)}_i(t,x) + (v^{(N)})^j \nabla_j u_i(t,x) &\\
+ (v^{(N)})^j \nabla_j v^{(N)}_i(t,x) + H^{(N)}_i(t,x,B^{(N)}(t,x))  &= -\partial_i q^{(N)} \label{vni}\\
\partial_t (B^{(N)})^i + u^j \nabla_j (B^{(N)})^i + (v^{(N)})^j \nabla_j (B^{(N)})^i &= - (v^{(N)})^j \nabla_j A^i \label{vni-2}\\
\nabla_i (v^{(N)})^i &= 0\label{vni-3}
\end{align}
with initial conditions 
\begin{equation}\label{vni-4}
v^{(N)}(0) = B^{(N)}(0)=0.
\end{equation}

We use the method of \emph{a priori} estimates, combined with the energy method.  Assume that we can obtain a smooth solution to the above system on some time interval $[0,T^{(N)})$ with $0 < T^{(N)} < T$.  Let $s$ be a large natural number, and define the energy
$$ E^{(N)}_s(t) := \sum_{k=0}^s \frac{1}{2} \int_{(\R/\Z)^d}|\nabla^k v^{(N)}(t,x)|_g^2 + |\nabla^k B^{(N)}(t,x)|_g^2\ dg,$$
then $E^{(N)}_s(0)=0$, where $\nabla^k v$ denote the iterated $k$-fold covariant derivative of a tensor $v$ (thus increasing the rank of the tensor by $k$) and $|v|_g^2 = \langle v, v \rangle_g$ denotes the norm squared of a tensor relative to the metric $g$, and $dg$ is the Riemannian volume form.  From the initial conditions we have $E^{(N)}_s(0)=0$.  The time derivative can be computed as
$$ 
\partial_t E^{(N)}_s = -\sum_{k=0}^s \sum_{i=1}^8 G_{k,i}$$
where
\begin{align*}
G_{k,1} &\coloneqq \int_{(\R/\Z)^d} \langle \nabla^k (v^{(N)})^i, \nabla^k (u^j \nabla_j v^{(N)}_i) \rangle_g\ dg \\
G_{k,2} &\coloneqq \int_{(\R/\Z)^d} \langle \nabla^k (v^{(N)})^i, \nabla^k ((v^{(N)})^j \nabla_j u_i) \rangle_g\ dg \\
G_{k,3} &\coloneqq \int_{(\R/\Z)^d} \langle \nabla^k (v^{(N)})^i, \nabla^k ((v^{(N)})^j \nabla_j v^{(N)}_i) \rangle_g\ dg \\
G_{k,4} &\coloneqq \int_{(\R/\Z)^d} \langle \nabla^k (v^{(N)})^i, \nabla^k H^{(N)}_i(t,x,B^{(N)}) \rangle_g\ dg \\
G_{k,5} &\coloneqq \int_{(\R/\Z)^d} \langle \nabla^k (v^{(N)})^i, \nabla^k \partial_i q^{(N)} \rangle_g\ dg \\
G_{k,6} &\coloneqq \int_{(\R/\Z)^d} \langle \nabla^k (B^{(N)})^i, \nabla^k (u^j \nabla_j (B^{(N)})^i) \rangle_g\ dg \\
G_{k,7} &\coloneqq \int_{(\R/\Z)^d} \langle \nabla^k (B^{(N)})^i, \nabla^k ((v^{(N)})^j \nabla_j (B^{(N)})^i) \rangle_g\ dg \\
G_{k,8} &\coloneqq \int_{(\R/\Z)^d} \langle \nabla^k (B^{(N)})^i, \nabla^k ((v^{(N)})^j \nabla_j A^i) \rangle_g\ dg.
\end{align*}
We now use the asymptotic notation $X = O(Y)$ or $X \lesssim Y$ to denote the bound $|X| \leq C Y$ where $C$ can depend on $d, s, g, u, F, A$ but is uniform in $N$ (in particular, the first $s$ derivatives of the Riemann curvature tensor are $O(1)$).  From the product rule and H\"older's inequality we easily see that
$$ G_{k,2} = O( E^{(N)}_s ).$$
The quantity $G_{k,1}$ cannot be immediately estimated in this fashion due to the possibility that $s+1$ derivatives fall on the second $v^{(N)}$ factor.  However, from the divergence-free nature of $u$ we have
$$ \int_{(\R/\Z)^d} \langle \nabla^k (v^{(N)})^i, u^j \nabla_j \nabla^k v^{(N)}_i \rangle_g\ dg = 0$$
and by subtracting this from $G_{k,1}$ (and noting that all the curvature terms arising are lower order) we conclude that
$$ G_{k,1} = O( E^{(N)}_s ).$$
To treat $G_{k,3}$, we similarly subtract off the identity
$$ \int_{(\R/\Z)^d} \langle \nabla^k (v^{(N)})^i, (v^{(N)})^j \nabla_j \nabla^k v^{(N)}_i \rangle_g\ dg = 0$$
and eventually conclude from the triangle inequality that
$$ |G_{k,3}| \lesssim \sum_{k=0}^s \sum_{k_1+k_2 = k+1: k_1,k_2 \leq k} \int_{(\R/\Z)^d} |\nabla^k v^{(N)}|_g |\nabla^{k_1} v^{(N)}|_g |\nabla^{k_2} v^{(N)}|_g\ dg.$$
One of the exponents $k_1,k_2$ will be less than $\frac{s+1}{2}$; if $s$ is large enough, this term can be bounded pointwise by $O((E^{(N)}_s)^{1/2})$ by Sobolev embedding.  From H\"older's inequality, we then conclude that
$$ G_{k,3} = O( (E^{(N)}_s)^{3/2} ).$$
We now turn to $G_{k,4}$.  From the fundamental theorem of calculus, \eqref{hnt}, and the fact that $F^{(N)}$ converges in the smooth topology to $F$ one has the bound
$$  H^{(N)}(t,x,B) = O(B) + o(1)$$
uniformly in $t,x$, where $o(1)$ denotes a quantity that goes to zero as $N \to \infty$; similarly for any given derivative of $H^{(N)}$.  From this and many applications of the chain rule (or Faa di Bruno formula) we conclude that
$$ |\nabla^k H^{(N)}_i(t,x,B^{(N)})|_g \lesssim \sum_{j=0}^k \sum_{k_1+\dots+k_j \leq k} |\nabla^{k_1} B^{(N)}|_g \dots |\nabla^{k_j} B^{(N)}|_g |B^{(N)}|_g + o(1).$$
All but at most one of the indices $k_1,\dots,k_j$ will be at most $\frac{s}{2}$, so by Sobolev embedding as before we conclude that
$$ |\nabla^k H^{(N)}_i(t,x,B^{(N)})|_g \lesssim (1 + (E^{((N)}_s)^{s/2}) \sum_{j=0}^s |\nabla^j B^{(N)}|_g + o(1),$$
and then from H\"older's inequality we have
$$ |G_{k,4}| \lesssim (1 + (E^{((N)}_s)^{s/2}) E^{(N)}_s + o((E^{(N)}_s)^{1/2}).$$
To treat $G_{k,5}$, we observe from the divergence-free nature of $v^{(N)}_i$ that
$$ \int_{(\R/\Z)^d} \langle \nabla^k \nabla_i (v^{(N)})^i, \nabla^k q^{(N)} \rangle_g\ dg = 0;$$
adding this to $G_{k,5}$ and commuting covariant derivatives and integrating by parts, then using H\"older's inequality, we conclude that
$$ |G_{k,5}| \lesssim (E^{(N)}_s)^{1/2} (\sum_{j=1}^k \int_{(\R/\Z)^d} |\nabla^j q^{(N)}|_g^2\ dx)^{1/2}.$$
On the other hand, taking the divergence of \eqref{vni} and using the divergence-free nature of $v^{(N)}_i$, we have
\begin{equation}\label{nij}
 \nabla^i (u^j \nabla_j v^{(N)}_i + (v^{(N)})^j \nabla_j u_i + (v^{(N)})^j \nabla_j v^{(N)}_i) + \nabla^i H^{(N)}_i(t,x,B^{(N)}(t,x)) = -\Delta q^{(N)}
\end{equation}
where $\Delta$ is the Laplace-Beltrami operator.  From elliptic regularity, the Leibniz\footnote{If one wishes, one can use the divergence-free nature of $v^{(N)}_i$ and $u_i$ to replace the terms on the left-hand side of \eqref{nij} involving second derivatives with lower order curvature terms, but this will not be necessary to close the argument here, as the only non-zero terms here are curvature terms and are thus already of lower order.} and chain rules, H\"older, and Sobolev, we then have the estimate
$$ (\sum_{j=1}^k \int_{(\R/\Z)^d} |\nabla^j q^{(N)}|_g^2\ dx)^{1/2} \lesssim (1 + (E^{(N)}_s)^{s/2}) (E^{(N)}_s)^{1/2} + o(1)$$ 
and thus
$$ |G_{k,5}| \lesssim (1 + (E^{((N)}_s)^{s/2}) E^{(N)}_s + o((E^{(N)}_s)^{1/2}).$$
By repeating the argument used to estimate $G_{k,1}, G_{k,2}, G_{k,3}$, we have
\begin{align*}
 G_{k,6} &= O( E^{(N)}_s ) \\
 G_{k,7} &= O( (E^{(N)}_s)^{3/2} ) \\
 G_{k,8} &= O( E^{(N)}_s ).
\end{align*}
Putting everything together, we conclude that
$$ 
|\partial_t E^{(N)}_s| \lesssim (1 + (E^{((N)}_s)^{s/2}) E^{(N)}_s + o((E^{(N)}_s)^{1/2})$$
and then a standard continuity argument (combined with the initial condition $E^{(N)}_s(0)=0$) then shows that for $N$ large enough one has
\begin{equation}\label{enst}
 E^{(N)}_s(t) = o(1)
\end{equation}
uniformly in $0 \leq t < T^{(N)}$.  In particular, the energy $E^{(N)}_s(t)$ does not blow up as $t$ approaches $T^{(N)}$, which by standard energy method local existence theory (adapting for instance the arguments in \cite[\S 3.2]{bertozzi}) implies (for $s$ large enough) that one can in fact solve the initial value problem \eqref{vni}-\eqref{vni-4} smoothly (and uniquely) all the way up to time $T$.  From \eqref{enst} and Sobolev embedding we then see that $v^{(N)}$ converges in the smooth topology to zero, and thus $u^{(N)}$ converges in the smooth topology to $u$ as required.
\end{proof}

\end{document}